\documentclass[final, nomarks]{dmtcs-episciences}
\usepackage{amsmath,amsthm,amssymb,amsfonts}
\theoremstyle{plain}
\newtheorem{theorem}{Theorem}
\newtheorem{lemma}[theorem]{Lemma}
\newtheorem{fact}[theorem]{Fact}
\newtheorem{proposition}[theorem]{Proposition}
\newtheorem{conjecture}[theorem]{Conjecture}

\theoremstyle{definition}
\newtheorem*{definition}{Definition}
\theoremstyle{remark}

\usepackage[utf8]{inputenc}
\usepackage[left=1in, right=1 in]{geometry}

\usepackage[round]{natbib}

\title{Antipowers in Uniform Morphic Words and the Fibonacci Word}
\author{Swapnil Garg\thanks{Work supported by NSF/DMS grant 1659047 and NSA grant H98230-18-1-0010.}}
\affiliation{Massachusetts Institute of Technology, Cambridge, MA}
\date{June 2021}

\received{2021-01-27}

\revised{2021-06-22}

\accepted{2021-10-20}
\begin{document}
\publicationdetails{23}{2021}{3}{14}{7134}

\maketitle

\begin{abstract}
Fici, Restivo, Silva, and Zamboni define a $k$-antipower to be a word composed of $k$ pairwise distinct, concatenated words of equal length. Berger and Defant conjecture that for any sufficiently well-behaved aperiodic morphic word $w$, there exists a constant $c$ such that for any $k$ and any index $i$, a $k$-antipower with block length at most $ck$ starts at the $i$-th position of $w$. They prove their conjecture in the case of binary words, and we extend their result to alphabets of arbitrary finite size and characterize those words for which the result does not hold. We also prove their conjecture in the specific case of the Fibonacci word.
\end{abstract}

\section{Introduction}
This paper settles certain cases of a conjecture posed by \cite{berdef} concerning antipowers, first introduced by \cite{fici}. They define a $k$-antipower to be a word that is the concatenation of $k$ pairwise distinct blocks of equal length. For example, $011000$ is a $3$-antipower, as $01, 10, 00$ are pairwise distinct. A variety of papers have been produced on the subject in the following years including \cite{inform, amanda, colin, ficipostic, marisa, represent, shyam}, with \cite{colin, marisa, shyam} finding bounds on antipower lengths in the Thue-Morse word.

Clearly one can construct periodic words without long antipowers, but what about other words? An \textit{aperiodic} infinite word is defined as a word with no periodic suffix, and an infinite word $w$ is \textit{recurrent} if every finite substring of $w$ appears in $w$ infinitely many times as a substring. We say $w$ is \textit{uniformly recurrent} if for every integer $a$, there is a larger integer $b$ such that every length-$a$ substring of $w$ appears as a substring in every length-$b$ substring of $w$. Here, substrings are sequences of consecutive letters in a word. Fici et al.\ asked whether such words can avoid long antipowers, and came to the following conclusion:

\begin{theorem}[\cite{fici}]
\label{thm:1}
\begin{itemize}
    \item Every infinite aperiodic word contains a $3$-antipower.
    \item There exist infinite aperiodic words avoiding $4$-antipowers.
    \item There exist infinite recurrent aperiodic words avoiding $6$-antipowers.
\end{itemize}
\end{theorem}

Berger and Defant complete this question with the following theorem:

\begin{theorem}[\cite{berdef}]
\label{thm:2}
Every infinite aperiodic recurrent word contains a $5$-antipower.
\end{theorem}

Berger and Defant then investigated whether more restrictions on words can force the inclusion of large antipowers. Specifically, they look at morphic words. We denote the size of the finite alphabet $A$ as $m$. Infinite words over $A$ are infinite to the right, so prefixes of infinite words are finite, while suffixes are infinite. Let $A^*$ denote the set of finite words over $A$, let $A^{\omega}$ be the set of infinite words over $A$, and let $A^{\infty}=A^* \cup A^{\omega}$. A \textit{morphism} of $A^{\infty}$ is a function $\mu: A^{\infty} \rightarrow A^{\infty}$ such that for words $u, v$ with $u$ finite, $\mu(uv)=\mu(u)\mu(v)$. So, $\mu$ is determined by its values on the letters in $A$. Let $|w|$ denote the length of a word $w$. An $r$-\textit{uniform} morphism is one where $|\mu(a)|=r$ for all $a \in A$, and a morphism that is $r$-uniform for some $r$ is called \textit{uniform}. For the rest of this paper, we assume $r \ge 2$ for an $r$-uniform morphism.

A morphism $\mu$ is called \textit{prolongable} on $a$ if $\mu(a)$ starts with the letter $a$. If we have such a $\mu$ and $a$, then repeatedly applying $\mu(a)$ results in a limiting infinite word $\mu^{\omega}(a)$. We work with an infinite pure morphic word $w$, which is equal to $\mu^{\omega}(a)$ for some $a \in A$ and $r$-uniform morphism $\mu$ that is \textit{prolongable} on $a$. As $\mu(a)$ begins with $a$, we have that $\mu^n(a)$ is a prefix of $\mu^{n+1}(a)$, so $\mu^{\omega}(a)$ is well-defined as the limit of $\mu^n(a)$ as $n$ goes to $\infty$.

For example, consider the Thue-Morse word $\textbf{t}$, defined as $\mu^{\omega}(0)$ for $\mu(0)=01, \mu(1)=10$. Here $\mu$ is a $2$-uniform morphism, which is prolongable on $0$ (and $1$). We have that $\textbf{t}=0110100110010110\cdots$.

\begin{conjecture}[\cite{berdef}]
\label{conj:3}
Every sufficiently well-behaved morphic word $w$ has a constant $c$ such that for any $k$, a $k$-antipower with block length at most $ck$ starts at each index of $w$.
\end{conjecture}

They settle Conjecture~\ref{conj:3} in a certain special case, proving the following theorem:

\begin{theorem}[\cite{berdef}]
\label{thm:4}
Every aperiodic, uniformly recurrent binary word $w$ generated by a uniform morphism has a constant $c$ such that for any $k$, a $k$-antipower with block length at most $ck$ starts at each index of $w$.
\end{theorem}

In this paper, we first extend this result to alphabets of arbitrary size. We then prove Conjecture~\ref{conj:3} in the case of the Fibonacci word, a special case of a word generated by a non-uniform morphism. Specifically, we prove the following two theorems:

\begin{theorem}
\label{thm:5}
Suppose $w$ is an aperiodic, uniformly recurrent word generated by a uniform morphism (over any finite size alphabet). Then at every index of $w$ starts a $k$-antipower with block length at most $ck$ for some constant $c$ only depending on $w$.
\end{theorem}

\begin{theorem}
\label{thm:6}
There is a constant $c \le \frac{4}{\sqrt{5}}\phi \approx 2.89$ such that for any $k$, at any index of the Fibonacci word starts a $k$-antipower with block length at most $ck$.
\end{theorem}

Theorem~\ref{thm:6} is the first instance of a proof of Conjecture~\ref{conj:3} in the case of a word generated by a non-uniform morphism.

Theorem~\ref{thm:5} was later proved independently by \cite{postic}; our proof is an alternative one, as it does not rely on results about recognizability.

\section{Antipowers in Uniform Morphic Words}
A \textit{conjugate} of a word $w$ is a cyclic rotation of $w$, that is, any word $vu$ if $w=uv$ for words $u, v$. A word is \textit{primitive} if it equals none of its conjugates, i.e. it is not periodic. For any word $v$, let $v_{[i, j]}$ denote the substring of $v$ starting at index $i$ and ending one before $j$, where $v$ is $0$-indexed. Also, let $v_i$ be the $i$-th letter in $v$. For example, if $v=01101001$, then $v_{[2, 6]}=1010$ is the string consisting of the middle four letters of $v$, and $v_0=0$.

We use the following fact about the complexity of infinite words.
\begin{lemma}[\cite{algcombo}]
\label{lem:7}
Let $w$ be an infinite aperiodic word. Then, for all positive integers $k$, the number of distinct substrings of length $k$ in $w$ is at least $k+1$.
\end{lemma}

If an aperiodic word $w$ has exactly $k+1$ substrings of each length $k$, then $w$ is called \textit{Sturmian}. 

We are now in a position to prove the conjecture posed by Berger and Defant in the case of aperiodic, uniformly recurrent words generated by a uniform morphism. We use a method similar to their proof of Theorem~\ref{thm:4}, which solves Conjecture~\ref{conj:3} for such words over a binary alphabet.

\begin{lemma}
\label{lem:8}
Let $w$ be an aperiodic, uniformly recurrent infinite word generated by an $r$-uniform morphism $\mu$. Let $t$ be a substring of $w$ such that every two letter substring of $w$ is a substring of $t$. Let $s$ be a substring of $w$ such that $s=ftg$ for some letters $f, g \in A$, so that $s$ contains $t$ and is one letter longer on each side. Fix $n$, and suppose that $\mu^n(s)$ is a substring of $w$, so $\mu^n(s)=w_{[\gamma, \gamma+s\cdot r^n]}$ for some $\gamma$. If the remainder when $\gamma$ is divided by $r^n$ is $i$, then $\gcd(i, r^n)  > \frac{r^n}{m^2}$.
\end{lemma}
\begin{proof}
Let $S = \{ \mu^n(a) | a \in A\}$.

First, we will prove that if $a, b \in A$ are such that $ab$ is a substring of $w$, then $(\mu^n(ab))_{[r^n-i, 2r^n-i]}$ is in $S$. (*)

Since $ab$ is a substring of $t$, which is in $s$, we have that $ab=s_{[c, c+2]}$ for some integer $c$, meaning that $\mu^n(ab)$ is a substring of $\mu^n(s)$ with 
$$\mu^n(ab)=\mu^n(s)_{[cr^n, (c+2)r^n]}.$$ Then, since $\mu^n(s)=w_{[\gamma, \gamma+s\cdot r^n]}$, we have that
\begin{align*}(\mu^n(ab))_{[r^n-i, 2r^n-i]}=&\mu^n(s)_{[(c+1)r^n-i, (c+2)r^n-i]}\\=&w_{[(\gamma-i)+(c+1)r^n, (\gamma-i)+(c+2)r^n]}\\=&\mu^n(w_{(\gamma-i)/r^n + (c+1)}) \in S.\end{align*}

Now, we claim that for all positive integers $p$, if the remainder when $p\cdot i$ is divided by $r^n$ is $x$, then for any substring $ab$ of $w$ for $a, b \in A$, we have $(\mu^n(ab))_{[r^n-x, 2r^n-x]} \in S$. We prove this claim by induction on $p$. We have already proved the base case $p=1$. Suppose that the statement is true for some $p$ with $pi$ having remainder $x$ when divided by $r^n$.

Fix $a, b \in A$ such that $ab$ is a substring of $w$.\\

\noindent\textbf{Case 1}: $x+i \le r^n$. Take $c \in A$ such that $cab$ is a substring of $s$. Then there are letters $d, e \in A$ such that $(\mu^n(cab))_{[r^n-i, 3r^n-i]}=\mu^n(de)$ by the claim (*). Therefore, \begin{align*}(\mu^n(ab))_{[r^n-x-i, 2r^n-x-i]}=&(\mu^n(cab))_{[2r^n-x-i, 3r^n-x-i]}\\=&(\mu^n(de))_{[r^n-x, 2r^n-x]} \in S.\end{align*}

\noindent\textbf{Case 2}: $x+i > r^n$. Take $c \in A$ such that $abc$ is a substring of $s$. Then there are letters $d, e \in A$ such that $(\mu^n(abc))_{[r^n-i, 3r^n-i]}=\mu^n(de)$ by the claim (*). Therefore, \begin{align*}(\mu^n(ab))_{[r^n-x-i+r^n, 2r^n-x-i+r^n]}=&\mu^n((abc))_{[2r^n-x-i, 3r^n-x-i]}\\=&(\mu^n(de))_{[r^n-x, 2r^n-x]} \in S.\end{align*}

Now suppose that for sake of contradiction, $\gcd(i, r^n)=j \le \frac{r^n}{m^2}$. Then, we have that $w_{[cr^n-p\cdot j, (c+1)r^n-p\cdot j]} \in S$ for any $p, c$ such that $cr^n-pj \ge 0$. Hence, the total number of distinct substrings of length $r^n$ in $w$ is at most $m + (j-1)m^2 \le r^n$, since we can write any such substring as either $\mu^n(a)$ for $a \in A$ or $(\mu^n(ab))_{[r^n-x, 2r^n-x]}$ for $a, b \in A$ and $0 < x < j$. However, if $w$ is aperiodic, then it must have at least $r^n+1$ distinct substrings of length $r^n$ by Lemma~\ref{lem:7}.
\end{proof}

\begin{proof}[of Theorem~\ref{thm:5}]

Take $n$ such that $\frac{r^n}{m^2} \ge k$. Let $s$ be as in the statement of Lemma~\ref{lem:8}. Because $w$ is uniformly recurrent, there is a constant $y$ such that $s$ is a substring of any length-$y$ substring of $w$. Then, consider the word starting at a given index $a$ with $k$ blocks of size $r^n \cdot y + 2r^n-1$. Since each block covers at least $y$ blocks of size $r^n$ that are in $\mu^n(A)$, each block contains a copy of $\mu^n(s)$ that starts at an index divisible by $r^n$. Suppose that the $i$-th block and the $j$-th block are equal.

Then, since the $j$-th block starts at an index shifted to the left by $(j-i)$ modulo $r^n$ compared to the $i$-th block, the $j$-th block must have a copy of $\mu^n(s)$ starting at an index congruent to $(j-i)$ modulo $r^n$. That is, we have $$w_{[cr^n-(j-i), (c+1)r^n-(j-i)]} = \mu^n(s)$$ for some integer $c$. But $|j-i| < k \le \frac{r^n}{m^2}$, so $\gcd(j-i, r^n) \le |j-i| < \frac{r^n}{m^2}$, which is impossible by Lemma~\ref{lem:8}. Therefore, we have constructed a $k$-antipower starting at every index, and we are done.
\end{proof}

Now, we aim to classify the infinite aperiodic uniformly recurrent words that arise from a uniform morphism. Classifying the uniformly recurrent words is easier than classifying aperiodic words, as we see below.

\begin{lemma}[Proposition 5.2, \cite{queffelec}]
\label{lem:9}
Let $w=\mu^{\omega}(0)$ over an alphabet $A$ of size $m$ for an $r$-uniform morphism $\mu$ be such that $w$ contains all letters in $A$. Then $w$ is uniformly recurrent if and only if $\mu^{m-1}(a)$ contains $0$ for all letters $a$ in $w$.
\end{lemma}
%\begin{proof}
%If we create a directed graph between the letters in $w$ where $a$ points to $b$ when $b \in \mu(a)$, then the set of letters in $\mu^n(a)$ is just the set of letters that are a walk of distance $n$ away from $a$. Because $0 \in \mu(0)$, we have that $0$ points to itself in the graph. Therefore, if $0\in \mu^k(a)$, then $0 \in \mu^j(a)$ for $j > k$. For each letter $a$ in $w$, either there is a path from $a$ to $0$, which then has length at most $m-1$, or there is no path. In the first case, $0 \in \mu^m(a)$, and in the second, for arbitrarily large $k$, $\mu^k(a)$ can be arbitrarily long and have no $0$, making $w$ not uniformly recurrent.

%So, we need to prove that $0$ being in $\mu^{m-1}(a)$ for all letters $a \in w$, means that $w$ is uniformly recurrent. For any substring $s$ of $w$, there is some number $k$ such that $s \in \mu^{k}(0)$. The number of letters between any two consecutive $0$s is at most $2r^{m-1}-2$, so the number of letters between any two consecutive words $s$ is at most $r^k(2r^{m-1}-2)$. Thus, $s$ appears as a substring in any length-$(r^k(2r^{m-1}-2)+2|s|-1)$ substring of $w$, so $w$ is uniformly recurrent and we are done.
%\end{proof}

Having classified uniformly recurrent words, we turn to aperiodic words. We will determine criteria for a word to not be aperiodic.
\begin{definition}
A word $w$ is \textit{eventually periodic} if there is some integer $n$ such that deleting the first $n$ letters of $w$ makes it periodic.
\end{definition}

\begin{lemma}
\label{lem:10}
Suppose that $w$ is an eventually periodic word generated by an injective $r$-uniform morphism $\mu$. Then, the period of $w$ is not divisible by $r$.
\end{lemma}
\begin{proof}
Suppose that the period of $w$ is divisible by $r$ and equals $kr$ for some $k$. Then, if we start far enough along in the word and take $w_{[nr, nr+kr]}$ for a large enough integer $n$, we get that $w_{[nr, nr+kr]}$ is a repeating unit of $w$ and is equal to $\mu(a_1)\mu(a_2)\cdots\mu(a_k)$ for letters $a_1, a_2, \dots, a_k$. But then $a_1a_2\cdots a_k$ is a repeating unit of $w$ since $w=\mu^{-1}(w)$, contradicting the minimality of the period $kr$.
\end{proof}

\begin{lemma}
\label{lem:11}
Suppose that $w$ is an eventually periodic, recurrent word. Then $w$ is periodic.
\end{lemma}
\begin{proof}
For sake of contradiction, suppose that $w$ is only eventually periodic, with period $\ell$ and starting index $i>0$. Let $s=w_{[i-1, i+l]}$, with length $\ell+1$. Because $w$ is recurrent, $s$ must appear infinitely many times in $w$ as a substring, so it appears in the periodic part of $w$ as a substring. But since that part has period $\ell$, the first and last letters in $s$ must be the same, contradicting the fact that $i$ is the starting point of the periodic part of $w$. So, $w$ is periodic.
\end{proof}

\begin{lemma}
\label{lem:12}
Suppose that $w$ is a periodic infinite word with minimal repeating unit $t$, i.e., $t$ is the smallest word such that $w=t^{\omega}$. Then $t$ is primitive.
\end{lemma}
\begin{proof}
Suppose that $t$ equals one of its conjugates. Let $t$ have length $\ell$, and equal itself shifted by $i$. Then, $t_0=t_i=t_{2i}=\cdots$, so $t_0=t_x$ for any $x$ that is the remainder of an integer multiple of $i$ modulo $l$. If $\gcd(i, \ell)=\ell'$, we have that $t_{[0, \ell']}=t_{[\ell', 2\ell']}=\cdots$ so $t$ itself is repeating with a period $\ell'$, which is impossible as $t$ is the minimal repeating unit of $w$.
\end{proof}

\begin{theorem}
\label{thm:13}
Let $w$ be a periodic word generated by an injective $r$-uniform morphism $\mu$. Then, the minimal repeating unit of $w$ has no letter appearing twice.
\end{theorem}
\begin{proof}

Suppose we have a periodic word $w$ generated by an $r$-uniform morphism $\mu$ applied to $0$. Let the period of $w$ be denoted $\ell$. Then, $r\ell$ is a non-minimal period for $w$. If $\frac{r\ell}{k}=\text{lcm}(r, \ell)< r\ell$, then $\frac{r\ell}{k}$ is a non-minimal period for $w$ for some $k$ dividing $\ell$. Then since $w=\mu(w)$ is periodic with period $r\frac{\ell}{k}$, the word $\mu^{-1}(w)=w$ must be periodic with period $\frac{\ell}{k}$ as the mapping $\mu$ is injective.

Hence, $\frac{\ell}{k}$ is a period for $w$, contradicting the minimality of $\ell$. So, $\ell$ and $r$ are relatively prime. If $\ell=r$ then we are done. Otherwise, let the length $\ell$ repeating unit be $t$.
We have two cases: the period $\ell$ is either less than or greater than $r$.\\

\noindent\textbf{Case 1}: $\ell < r$. Suppose that $t$ has a duplicate letter, say $w_i=w_j$ for $0 \le i < j < \ell$. Then, since $\mu(w)=w$, we have that $w_{[ir, (i+1)r]}=w_{[jr, (j+1)r]}$. In particular, $w_{[ir, ir+l]}=w_{[jr, jr+l]}$. However, since $t$ is primitive, we must have that $ir \equiv jr \mod \ell$ or $(j-i)r \equiv 0 \mod \ell$. But $0 < j-i < \ell$ and $r$ is relatively prime to $l$, so this is impossible. Therefore $t$ has no duplicate letters.\\

\noindent\textbf{Case 2}: $\ell > r$. We generalize the previous case. Suppose that $t$ has a duplicate letter, say $w_i=w_j$ for $0 \le i < j < \ell$. Then, for every $k$, we have that $w_{[ir^k, (i+1)r^k]}=w_{[jr^k, (j+1)r^k]}$. In particular, we can take $k$ such that $r^k \ge \ell$. Then, since $t$ is primitive, we must have $ir^k \equiv jr^k \mod \ell$, or $(j-i)r^k \equiv 0 \mod \ell$, which is impossible as $r^k$ is relatively prime to $l$ and $0 < j-i < l$.
\end{proof}

Therefore, for a periodic word generated by an injective uniform morphism, the repeating unit must consist of distinct letters. If the repeating unit has length $l$, then $w$ has $l$ distinct letters. Any periodic word starting with $0$ and consisting of a repeating unit of $l$ distinct letters can be generated by an $r$-uniform morphism as long as $r$ is relatively prime to $l$. For example, the word $012301230123\cdots$ can be written as $\mu^{\omega}(0)$ with $\mu(0)=01230, \mu(1)=12301, \mu(2)=23012, \mu(3)=30123$. So, except for a small class of exceptional words that we have characterized, all words generated by a injective uniform morphism are aperiodic and uniformly recurrent, and therefore satisfy the hypothesis of Theorem~\ref{thm:5}. Note that we have not classified which words generated by a noninjective uniform morphism satisfy the hypothesis of Theorem~\ref{thm:5}.

\section{Antipowers in the Fibonacci Word}
We prove that the Fibonacci word $\textbf{f}$, which is equal to $\varphi^{\omega}(0)$ for $\varphi(0)=01, \varphi(1)=0$ and thus pure morphic but not generated by a uniform morphism, also satisfies Conjecture~\ref{conj:3}. Let $\phi=\frac{1+\sqrt{5}}{2}$. An alternate characterization of the Fibonacci word is given below by the following well-known fact.

\begin{fact}
\label{fact:14}
The $n$-th digit in the Fibonacci word can be written as $2-(\lfloor (n+2)\phi \rfloor-\lfloor (n+1)\phi \rfloor)$.
\end{fact}

Let the Fibonacci sequence be defined as $F_1=F_2=1$ and $F_n=F_{n-1}+F_{n-2}$ for $n \ge 3$.

\begin{lemma}
\label{lem:15}
Modulo $1$, the real number $F_n \phi$ is congruent to $-(-\phi)^{-n}$.

\end{lemma}
\begin{proof}
By Binet's formula, $F_n=\frac{1}{\sqrt{5}}(\phi^n-(-\phi)^{-n})$ is an integer. So, \begin{align*}F_n\phi=&\frac{1}{\sqrt{5}}(\phi^{n+1}+(-\phi)^{-n+1})\\=&F_{n+1}+\frac{1}{\sqrt{5}}((-\phi)^{-n+1}+(-\phi)^{-n-1})\\=&F_{n+1}-(-\phi)^{-n}.\end{align*} \end{proof}

Denote the fractional part of $x$ as $\{x\}$.

\begin{proposition}
\label{prop:16}
Fix a positive integer $n$. Given a positive integer $\ell$, if $$\phi^{-n+1}\le \min(\{\ell \cdot 2F_n\phi\}, 1-\{\ell \cdot 2F_n\phi\}),$$ then $$\emph{\textbf{f}}_{[x, x+2F_n]} \neq \emph{\textbf{f}}_{[x+\ell \cdot 2F_n, x+\ell \cdot 2F_n+2F_n]}$$ for all indices $x \in \mathbb{N}$.
\end{proposition}

\begin{proof}
For a block $\textbf{f}_{[x, x+2F_n]}$, consider the fractional parts of the $2F_n+1$ numbers $(x+1)\phi, (x+2)\phi, \dots, (x+2F_n+1)\phi$. Whether the fractional part of $(i+2)\phi$ is greater than or less than the fractional part of $(i+1)\phi$ determines whether or not the $i$-th digit of $\textbf{f}$ is $0$ or $1$ by Fact~\ref{fact:14}.

Now, when we add $\ell \cdot 2F_n$ to the block $[x, x+2F_n)$ for some positive integer $\ell$, we are shifting the numbers $(x+1)\phi, \dots, (x+2F_n+1)\phi$ by $\ell \cdot 2F_n\phi$. So, we are adding the fractional part of $\ell \cdot 2F_n\phi$ to the numbers $\{(x+1)\phi\}, \{(x+2)\phi\}, \dots, \{(x+2F_n+1)\phi\}$, and then subtracting $1$ from the numbers that are now at least $1$. If $\{(x+\ell \cdot 2F_n)\phi\}=\{x\phi\}+\{(\ell \cdot 2F_n)\phi\}-1$, then we say that $\{x\phi\}$ \textit{wraps around}. Note that $\{x\phi\}$ wraps around if and only if $\{x\phi\} \in [1-\{\ell \cdot 2F_n \phi\}, 1)$.

If both $\{(i+2)\phi\}$ and $\{(i+1)\phi\}$ wrap around or both don't wrap around when adding $\ell\cdot 2F_n \phi$, then the $(i+\ell \cdot 2F_n)$-th digit is the same as the $i$-th digit; otherwise, it is different. If the two digits are the same for all $i$ with $x \le i < x+2F_n$, then either every such fractional part $\{i\phi\}$ for $x+1 \le i \le x+2F_n+1$ wraps around, meaning that they all belong to $\{x\phi\} \in [1-\{\ell \cdot 2F_n \phi\}, 1)$, or every such fractional part doesn't wrap around, meaning that they all belong to $[0, 1-\{\ell \cdot 2F_n \phi\})$. We will show that neither of these two conditions are possible under the assumption on $\ell$, which will prove the proposition.

We consider the set of points $S=\{\{(x+1)\phi\}, \{(x+2)\phi\}, \dots, \{(x+2F_n+1)\phi\}\}$ on the circle $[0, 1]$ where we identify $0$ and $1$. If all elements in $S$ wrap around, then the largest gap between two consecutive points in $S$ is at least $1-\{\ell \cdot 2F_n\phi\}$, as the interval $[0, 1-\{\ell \cdot 2F_n\phi\})$ contains no points in $S$. If no element in $S$ wraps around, then the largest gap between two consecutive points in $S$ is at least $\{\ell \cdot 2F_n\phi\}$, as the interval $[1-\{\ell \cdot 2F_n\phi\}, 1)$ contains no points in $S$. Thus the largest gap between two consecutive points of $S$ being at least $\min(\{\ell \cdot 2F_n\phi\}, 1-\{\ell \cdot 2F_n\phi\})$ is a necessary (but not sufficient) condition for the $i$-th and $(i+\ell \cdot 2F_n)$-th digits of the Fibonacci word to be the same for all $x \le i < x+2F_n$.

Now, we claim that the largest gap between any two consecutive points in $S$ is at most $\phi^{-n+1}$. This is because if $i \le x+F_n+1$, then the distance between $\{i \phi\}$ and $\{(i+F_n)\phi\}$ is $\phi^{-n}$, and the distance between $\{i\phi\}$ and $\{(i+F_{n-1})\phi\}$ is $\phi^{-n+1}$. In fact, the residue of $i\phi$ modulo $1$ is between those of $(i+F_n)\phi$ and $(i+F_{n-1})\phi$. Similarly, if we look at $i \ge x+F_n+1$, then $i\phi$ is close to and between $(i-F_n)\phi$ and $(i-F_{n-1})\phi$ modulo $1$. Therefore, every point in $S$ has another point in $S$ to its left and right at most $\phi^{-n+1}$ away, so the largest gap between two consecutive points of $S$ is at most $\phi^{-n+1}$.

Then if $\phi^{-n+1}\le \min(\{\ell \cdot 2F_n\phi\}, 1-\{\ell \cdot 2F_n\phi\})$, the largest gap between two consecutive points in $S$ is not at least $\min(\{\ell \cdot 2F_n\phi\}, 1-\{\ell \cdot 2F_n\phi\})$, so the necessary condition for the blocks to be equal is not satisfied and $$\textbf{f}_{[x, x+2F_n]} \neq \textbf{f}_{[x+\ell \cdot 2F_n, x+\ell \cdot 2F_n+2F_n]}.$$
\end{proof}

\begin{proposition}
\label{prop:17}
At any index in \emph{\textbf{f}} and any positive integer $n$, there is an $\lfloor F_n \frac{\sqrt{5}}{2} \rfloor$-antipower starting at that index with block length $2F_n$.
\end{proposition}

\begin{proof}
Let $x$ be any index in \textbf{f}. We consider the $\lfloor \frac{\sqrt{5}}{2}F_n\rfloor$ blocks of the form $\textbf{f}_{[x+\ell \cdot 2F_n, x+(\ell+1)\cdot 2F_n]}$, where $0 \le \ell <  \lfloor \frac{\sqrt{5}}{2}F_n\rfloor$, and show that they are pairwise different. The distance between any two blocks is $m \cdot 2F_n$ for some $0 \le m < \lfloor \frac{\sqrt{5}}{2}F_n\rfloor$. We aim to show that all such $m$ satisfy the hypothesis of Proposition~\ref{prop:16}, namely that $$\phi^{-n+1}\le \min(\{m \cdot 2F_n\phi\}, 1-\{m \cdot 2F_n\phi\}).$$

We have that $\phi^{-n+1} \le 2\phi^{-n}=\min(\{2F_n\phi\}, \{(1-2F_n\phi)\})$, with the latter equality by Lemma~\ref{lem:15}. Then for $m \le \lfloor \frac{1}{2\phi^{-n}} \rfloor $, we have
$$\min(\{m \cdot 2F_n\phi\}, \{(1-m \cdot 2F_n\phi)\})=\min(m \cdot 2\phi^{-n}, 1-m \cdot 2\phi^{-n}).$$
Since for $m \le \lfloor \frac{1}{2\phi^{-n}} \rfloor-1$ we have $m \cdot 2\phi^{-n}>\phi^{-n+1}$ and $1-m \cdot 2\phi^{-n} > 2\phi^{-n}>\phi^{-n+1}$, the hypothesis of Proposition~\ref{prop:16} is indeed satisfied for such $m$.

Therefore, at every index, there is a length-$\lfloor \frac{\phi^n}{2} \rfloor$ antipower starting at that index with block length $2F_n$. We have $\lfloor \frac{\phi^n}{2} \rfloor= \lfloor \frac{\sqrt{5}}{2}F_n+\frac{(-\phi)^{-n}}{2}\rfloor$. Now, if $F_n$ is even, then the distance from $\frac{\sqrt{5}}{2}F_n=(\phi-\frac{1}{2})F_n$ and the nearest integer is $\phi^{-n}$, and if $F_n$ is odd, then the distance is $\frac{1}{2}-\phi^{-n}$. In particular, the distance from $\frac{\sqrt{5}}{2}F_n$ to the nearest integer is greater than $\frac{\phi^{-n}}{2}$, which means that 
$\lfloor \frac{\sqrt{5}}{2}F_n+\frac{(-\phi)^{-n}}{2}\rfloor = \lfloor \frac{\sqrt{5}}{2}F_n\rfloor,$ and we are done. \end{proof}

We now give the proof that there is a linear bound on antipowers in the Fibonacci word.

\begin{proof}[of Theorem~\ref{thm:6}]
Let $n$ be the smallest integer such that $\lfloor F_n \frac{\sqrt{5}}{2} \rfloor$ is at least $k$. Then, starting at any given index $a$ of \textbf{f}, there is a $k$-antipower with block length $2F_n$ by Proposition~\ref{prop:17}. We have $F_n=\phi F_{n-1}+(-\phi)^{-n}$, and that $\lfloor F_{n-1}\frac{\sqrt{5}}{2} \rfloor \le k-1$, so \begin{align*}\frac{2F_n}{k} =& \frac{2F_n}{F_n\sqrt{5}/2}\frac{F_n}{F_{n-1}}\frac{F_{n-1}\sqrt{5}/2}{k}\\=& \frac{4}{\sqrt{5}}\left(\phi + \frac{(-\phi)^{-n}}{F_{n-1}}\right)\frac{F_{n-1}\sqrt{5}/2}{k}\\=& \frac{4}{\sqrt{5}}\phi\left(1 - \frac{(-\phi)^{-n-1}}{F_{n-1}}\right)\frac{F_{n-1}\sqrt{5}/2}{k}.\end{align*} Since $F_{n-1}\sqrt{5}/2$ is less than and at least $(\phi)^{-n+1}$ away from $k$, and $\frac{F_{n-1}\sqrt{5}}{2} > F_{n-1}$, we have that $\frac{F_{n-1}\sqrt{5}/2}{k} < 1 + \frac{(\phi)^{-n+1}}{F_{n-1}}$, so $\frac{2F_n}{k} < \frac{4}{\sqrt{5}}\phi$.
\end{proof}

Recall that a word is Sturmian if there are exactly $k+1$ distinct substrings of length $k$ in $f$ for all $k$. As the Fibonacci word is a Sturmian word (see \cite{algcombo}), we cannot have a $k$-antipower starting at any index with block length less than $k-1$. So, if we let $\gamma_i(k)$ be the smallest block length that starts a $k$-antipower at index $i$, we have that for any $i$,

$$1 \le \liminf\limits_{k \rightarrow \infty} \frac{\gamma_i(k)}{k} \le \frac{4}{\sqrt{5}},$$ $$ 1 \le \limsup\limits_{k \rightarrow \infty}\frac{\gamma_i(k)}{k} \le \frac{4}{\sqrt{5}}\phi.$$

The reasoning for these bounds is as follows: for all $k$ and all indices $i$, by Theorem~\ref{thm:6} we have that there is a $k$-antipower starting at $i$ with block length at most $\frac{4}{\sqrt{5}}\phi$, so $\gamma_i(k) \le \frac{4}{\sqrt{5}}{\phi}$ for all $k$. This fact implies the weaker condition of the upper bound on the $\limsup$ of $\frac{\gamma_i(k)}{k}$. Furthermore, for infinitely many $k$, specifically those $k$ equal to $\lfloor F_n \frac{\sqrt{5}}{2} \rfloor$, we have that there is a $k$-antipower with block length $2F_n$, and $\frac{2F_n}{\lfloor F_n \frac{\sqrt{5}}{2} \rfloor}$ approaches $\frac{4}{\sqrt{5}}$ for large $n$, giving the upper bound on the $\liminf$ of $\frac{\gamma_i(k)}{k}$. The lower bounds follow from the fact that $\gamma_i(k) \ge k-1$, so $\frac{\gamma_i(k)}{k} \ge 1-\frac{1}{k}$, a number that approaches $1$.

Based on empirical data, we conjecture the following.

\begin{conjecture}
\label{conj:18}
Let $F_n$ be an even Fibonacci number. Then, there is an $(F_n-1)$-antipower with block length $\frac{F_n}{2}+F_{n-1}$ that is a prefix of \emph{\textbf{f}}.
\end{conjecture}

If this conjecture were true, we would have the following:

$$1 \le \liminf\limits_{k \rightarrow \infty} \frac{\gamma_0(k)}{k} \le \frac{\sqrt{5}}{2}.$$

\acknowledgements
The author would like to thank Prof.\ Joe Gallian for organizing the Duluth REU where this research took place, as well as advisors Aaron Berger and Colin Defant for suggesting the problem. The author would also like to thank Defant, Gallian, and Shyam Narayanan for providing helpful comments on drafts of this paper.

\nocite{*}
\bibliographystyle{abbrvnat}
\bibliography{biblio}

\begin{thebibliography}{12}
\providecommand{\natexlab}[1]{#1}
\providecommand{\url}[1]{\texttt{#1}}
\expandafter\ifx\csname urlstyle\endcsname\relax
  \providecommand{\doi}[1]{doi: #1}\else
  \providecommand{\doi}{doi: \begingroup \urlstyle{rm}\Url}\fi

\bibitem[Badkobeh et~al.(2018)Badkobeh, Fici, and Puglisi]{inform}
G.~Badkobeh, G.~Fici, and S.~J. Puglisi.
\newblock Algorithms for anti-powers in strings.
\newblock \emph{Inform. Process. Lett.}, 137:\penalty0 57--60, 2018.
\newblock \doi{10.1016/j.ipl.2018.05.003}.

\bibitem[Berger and Defant(2020)]{berdef}
A.~Berger and C.~Defant.
\newblock On anti-powers in aperiodic recurrent words.
\newblock \emph{Adv. Appl. Math}, 121, 2020.
\newblock \doi{10.1016/j.aam.2020.102104}.

\bibitem[Burcroff(2018)]{amanda}
A.~Burcroff.
\newblock ($k$, $\lambda$)-anti-powers and other patterns in words.
\newblock \emph{Electron. J. Combin.}, 25:\penalty0 4.41, 2018.
\newblock \doi{10.37236/8032}.

\bibitem[Defant(2017)]{colin}
C.~Defant.
\newblock Anti-power prefixes of the \uppercase{T}hue-\uppercase{M}orse word.
\newblock \emph{Electron. J. Combin.}, 24:\penalty0 1.32, 2017.
\newblock \doi{10.37236/6321}.

\bibitem[Fici et~al.(2018)Fici, Restivo, Silva, and Zamboni]{fici}
G.~Fici, A.~Restivo, M.~Silva, and L.~Q. Zamboni.
\newblock Anti-powers in infinite words.
\newblock \emph{J. Combin. Theory, Ser. A}, 157:\penalty0 109--119, 2018.
\newblock \doi{10.1016/j.jcta.2018.02.009}.

\bibitem[Fici et~al.(2019)Fici, Postic, and Silva]{ficipostic}
G.~Fici, M.~Postic, and M.~Silva.
\newblock Abelian anti-powers in infinite words.
\newblock \emph{Adv. Appl. Math.}, 108:\penalty0 67--78, 2019.
\newblock \doi{10.1016/j.aam.2019.04.001}.

\bibitem[Gaetz(2021)]{marisa}
M.~Gaetz.
\newblock Anti-power $j$-fixes of the \uppercase{T}hue-\uppercase{M}orse word.
\newblock \emph{Discrete Math. Theor. Comput. Sci.}, 23\penalty0 (1), 2021.
\newblock \doi{10.46298/dmtcs.5483}.

\bibitem[Kociumaka et~al.(2019)Kociumaka, Radoszewski, Rytter, Straszyński,
  Waleń, and Zuba]{represent}
T.~Kociumaka, J.~Radoszewski, W.~Rytter, J.~Straszyński, T.~Waleń, and
  W.~Zuba.
\newblock Efficient representation and counting of antipower factors in words.
\newblock \emph{Language and Automata Theory and Applications}, pages 421--433,
  2019.
\newblock \doi{10.1007/978-3-030-13435-8_31}.

\bibitem[Lothaire(2002)]{algcombo}
M.~Lothaire.
\newblock \emph{Algebraic Combinatorics on Words}.
\newblock Cambridge University Press, Cambridge, 2002.
\newblock \doi{10.1017/CBO9781107326019.003}.

\bibitem[Narayanan(2020)]{shyam}
S.~Narayanan.
\newblock Functions on antipower prefix lengths of the
  \uppercase{T}hue-\uppercase{M}orse word.
\newblock \emph{Discrete Math.}, 343\penalty0 (2), 2020.
\newblock \doi{10.1016/j.disc.2019.111675}.

\bibitem[Postic(2019)]{postic}
M.~Postic.
\newblock Anti-powers in primitive uniform substitutions.
\newblock \emph{preprint arXiv:1908.10627}, 2019.

\bibitem[Queff\'elec(2010)]{queffelec}
M.~Queff\'elec.
\newblock \emph{Substitution Dynamical Systems -- Spectral Analysis}, volume
  1294.
\newblock Springer-Verlag, Berlin Heidelberg, 2nd edition, 2010.
\newblock \doi{10.1007/BFb0081890}.

\end{thebibliography}

\end{document}